\documentclass{article}
\usepackage{mathrsfs, amsmath, amsthm, amssymb, bm, mathtools, thm-restate}
\usepackage{graphicx, xcolor, tikz}
\usetikzlibrary{arrows.meta, intersections, calc, shapes}
\usepackage{caption, subcaption}
\usepackage{array}
\usepackage{cases}
\makeatletter
\def\th@plain{%
  \upshape 
}
\makeatother

\makeatletter
\renewenvironment{proof}[1][\proofname]{\par
  \pushQED{\qed}%
  \normalfont \topsep6\p@\@plus6\p@\relax
  \trivlist
  \item[\hskip\labelsep
        \bfseries
    #1\@addpunct{.}]\ignorespaces
}{%
  \popQED\endtrivlist\@endpefalse
}
\makeatother

\newtheorem{theorem}{Theorem}[section]
\newtheorem{lemma}{Lemma}[section]
\newtheorem{corollary}[theorem]{Corollary}

\newtheorem*{conjecture*}{Conjecture}
\theoremstyle{definition}
\newtheorem{definition}{Definition}
\newtheorem{remark}{Remark}

\usepackage[left=25mm, top=25mm, bottom=25mm, right=25mm]{geometry}
\usepackage[T1]{fontenc}
\usepackage[inline]{enumitem}
\usepackage[square, numbers, sort&compress]{natbib}

\usepackage[pdftex,%
  bookmarks=true,%
  bookmarksnumbered=true, 
  bookmarksopen=true, 
  plainpages=false,%
  pdfpagelabels,%
  colorlinks=true, 
  linkcolor=blue, 
  citecolor=blue,%
  anchorcolor=green,
  urlcolor= blue,
  breaklinks=true,
  hyperindex=true]{hyperref}

\newcommand{\etal}{et~al.\ }
\def\mad(#1){\mathrm{mad}(#1)}
\def\int(#1){\mathrm{int}(#1)}
\def\ext(#1){\mathrm{ext}(#1)}
\def\Int(#1){\mathrm{Int}(#1)}
\def\Ext(#1){\mathrm{Ext}(#1)}

\title{DP-4-coloring of planar graphs with some restrictions on cycles}
\author{Rui Li \qquad Tao Wang\footnote{\tt Corresponding
author: wangtao@henu.edu.cn} \\
{\small Institute of Applied Mathematics}\\
{\small Henan University, Kaifeng, 475004, P. R. China}}
\begin{document}
\date{}
\maketitle
\begin{abstract}
DP-coloring was introduced by Dvo\v{r}\'{a}k and Postle as a generalization of list coloring. It was originally used to solve a longstanding conjecture by Borodin, stating that every planar graph without cycles of lengths 4 to 8 is 3-choosable. In this paper, we give three sufficient conditions for a planar graph to be DP-4-colorable. Actually all the results (\autoref{MRA}, \ref{MRB} and \ref{MRC}) are stated in the ``color extendability'' form, and uniformly proved by vertex identification and discharging method. 
\end{abstract}

\section{Introduction}\label{sec1}
A {\bf list assignment} is a function $L$ which assigns a list $L(v)$ of colors to each vertex $v$. An {\bf $L$-coloring} of $G$ is a proper coloring $\phi$ of $G$ such that $\phi(v) \in L(v)$ for all $v \in V(G)$. A {\bf list $k$-assignment} is a list assignment $L$ with $|L(v)| \geq k$ for each vertex $v \in V(G)$. A graph $G$ is {\bf $k$-choosable} if it admits an $L$-coloring whenever $L$ is a list $k$-assignment. The {\bf list chromatic number} $\chi_{\ell}(G)$ is the least integer $k$ such that $G$ is $k$-choosable. List coloring is a generalization of ordinary proper coloring, thus we have that $\chi(G) \leq \chi_{\ell}(G)$. The Four Color Theorem states that every planar graph is $4$-colorable. It was proved that not every planar graph is $4$-choosable, see \cite{MR1386951, MR1431668}. Some sufficient conditions for planar graphs to be 4-choosable are well studied, see \cite{MR3648207,MR3638001,MR2586624,MR1687318,MR1852816,MR2538645}. 

Correspondence coloring is introduced by Dvo\v{r}\'{a}k and Postle \cite{MR3758240} as a generalization of list coloring, now it is also called DP-coloring. 
\begin{definition}\label{DEF1}
Let $G$ be a simple graph and $L$ be a list assignment for $G$. For each vertex $v \in V(G)$, let $L_{v} = \{v\} \times L(v)$; for each edge $uv \in E(G)$, let $\mathscr{M}_{uv}$ be a matching between the sets $L_{u}$ and $L_{v}$, and let $\mathscr{M} = \bigcup_{uv \in E(G)}\mathscr{M}_{uv}$, call it the {\bf matching assignment}. The matching assignment is called a {\bf $k$-matching assignment} if $L(v) = [k]$ for each $v \in V(G)$. A {\bf cover} of $G$ is a graph $H_{L, \mathscr{M}}$ (simply write $H$) satisfying the following two conditions: 
\begin{enumerate}[label =(C\arabic*)]
\item the vertex set of $H$ is the disjoint union of $L_{v}$ for all $v \in V(G)$; and
\item the edge set of $H$ is the matching assignment $\mathscr{M}$.
\end{enumerate}
\end{definition}
Note that the matching $\mathscr{M}_{uv}$ is not required to be a perfect matching between the sets $L_{u}$ and $L_{v}$, and possibly it is empty. It is easy to see that the induced subgraph $H[L_{v}]$ is an independent set for each vertex $v \in V(G)$. 

\begin{definition}
Let $G$ be a simple graph and $H$ be a cover of $G$. An {\bf $\mathscr{M}$-coloring} of $H$ is an independent set $\mathcal{I}$ in $H$ such that $|\mathcal{I} \cap L_{v}| = 1$ for each vertex $v \in V(G)$. The graph $G$ is {\bf DP-$k$-colorable} if for any list assignment $L(v) \supseteq [k]$ and any matching assignment $\mathscr{M}$, it has an $\mathscr{M}$-coloring. The {\bf DP-chromatic number $\chi_{\mathrm{DP}}(G)$} of $G$ is the least integer $k$ such that $G$ is DP-$k$-colorable. 
\end{definition}

For DP-3-colorable graphs, we refer the reader to \cite{MR3886261, MR3969021, MR4230514, MR3954054}; for ``weakly'' DP-3-colorable graphs, we refer the reader to \cite{MR3758240, MR3983123, Lu}. Dvo\v{r}\'{a}k and Postle \cite{MR3758240} observed that every planar graph is DP-5-colorable. In this paper, we concentrate on DP-4-coloring of planar graphs. 

Two cycles (or faces) are {\bf adjacent} if they have at least one common edge. Two cycles are {\bf intersecting} if they share at least one vertex. Liu and Li \cite{MR3881665} showed that every planar graph without 4-cycles adjacent to two triangles is DP-4-colorable. Kim and Ozeki \cite{MR3802151} showed that for each integer $k \in \{3, 4, 5, 6\}$, every planar graph without $k$-cycles is DP-4-colorable. Chen \etal \cite{MR3996735} showed that every planar graph without 4-cycles adjacent to 6-cycles is DP-4-colorable. Zhang and Li \cite{Zhang2019} showed that every planar graph without 5-cycles adjacent to 6-cycles is DP-4-colorable. Lv, Zhang and Li \cite{Lv2019} showed that every planar graph without intersecting 5-cycles is DP-4-colorable. 

Li and Wang \cite{Li2019} considered planar graphs without some mutually adjacent 3-, 4-, and 5-cycles. 
\begin{theorem}[Li and Wang \cite{Li2019}]\label{MRTHREE}
Every planar graph without subgraphs isomorphic to the configurations in \autoref{E} is DP-4-colorable. 
\end{theorem}
\begin{figure}[htbp]%
\centering
\subcaptionbox{\label{fig:subfig:a--}}{\begin{tikzpicture}
\coordinate (A) at (45:1);
\coordinate (B) at (135:1);
\coordinate (C) at (225:1);
\coordinate (D) at (-45:1);
\coordinate (H) at (90:1.414);
\draw (A)--(H)--(B)--(C)--(D)--cycle;
\draw (A)--(B);
\fill (A) circle (2pt)
(B) circle (2pt)
(C) circle (2pt)
(D) circle (2pt)
(H) circle (2pt);
\end{tikzpicture}}%
\hspace{1.5cm}
\subcaptionbox{\label{fig:subfig:b--}}{\begin{tikzpicture}
\coordinate (A) at (30:1);
\coordinate (B) at (150:1);
\coordinate (C) at (225:1);
\coordinate (D) at (-45:1);
\coordinate (H) at (90:1.414);
\coordinate (X) at (60:1.4);
\coordinate (Y) at (120:1.4);
\coordinate (T) at ($(H)!(A)!(X)$);
\coordinate (Z) at ($(A)!2!(T)$);
\draw (A)--(X)--(Z)--(H)--(Y)--(B)--(C)--(D)--cycle;
\draw (A)--(H)--(B);
\draw (H)--(X);
\fill (A) circle (2pt)
(B) circle (2pt)
(C) circle (2pt)
(D) circle (2pt)
(H) circle (2pt)
(X) circle (2pt)
(Y) circle (2pt)
(Z) circle (2pt);
\end{tikzpicture}}%
\hspace{1.5cm}
\subcaptionbox{\label{fig:subfig:c--}}{\begin{tikzpicture}
\coordinate (A) at (30:1);
\coordinate (B) at (150:1);
\coordinate (C) at (225:1);
\coordinate (D) at (-45:1);
\coordinate (H) at (90:1.414);
\coordinate (X) at (60:1.4);
\coordinate (Y) at (120:1.4);
\coordinate (T) at ($(A)!(H)!(X)$);
\coordinate (Z) at ($(H)!2!(T)$);
\draw (A)--(Z)--(X)--(H)--(Y)--(B)--(C)--(D)--cycle;
\draw (X)--(A)--(H)--(B);
\fill (A) circle (2pt)
(B) circle (2pt)
(C) circle (2pt)
(D) circle (2pt)
(H) circle (2pt)
(X) circle (2pt)
(Y) circle (2pt)
(Z) circle (2pt);
\end{tikzpicture}}
\caption{Forbidden configurations in \autoref{MRTHREE}.}
\label{E}
\end{figure}
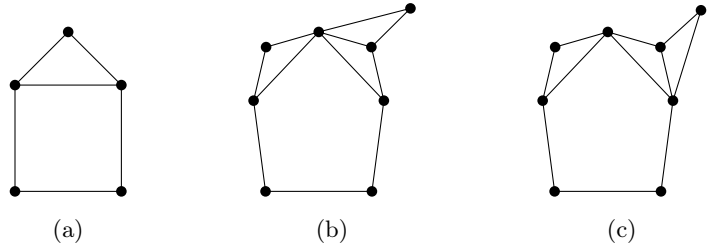

An edge $uv$ in $G$ is {\bf straight} in a $k$-matching assignment $\mathscr{M}$ if $(u, c_{1})(v, c_{2}) \in \mathscr{M}_{uv}$ satisfies $c_{1} = c_{2}$. An edge $uv$ in $G$ is {\bf full} in a $k$-matching assignment $\mathscr{M}$ if $\mathscr{M}_{uv}$ is a perfect matching. Let $W = w_{1}w_{2}\dots w_{m}$ with $w_{m} = w_{1}$ be a closed walk of length $m$ in $G$, a matching assignment is {\bf inconsistent} on $W$, if there exists $c_{i} \in L(w_{i})$ for $i \in [m]$ such that $(w_{i}, c_{i})(w_{i+1}, c_{i+1})$ is an edge in $\mathscr{M}_{w_{i}w_{i+1}}$ for $i \in [m-1]$ and $c_{1} \neq c_{m}$. Otherwise, the matching assignment is {\bf consistent} on $W$. 

\begin{lemma}[Dvo\v{r}\'{a}k and Postle \cite{MR3758240}]\label{ST}
Let $G$ be a graph with a $k$-matching assignment $\mathscr{M}$, and let $H$ be a subgraph of $G$. If for every cycle $\mathcal{Q}$ in $H$, the assignment $\mathscr{M}$ is consistent on $\mathcal{Q}$ and all edges of $\mathcal{Q}$ are full, then we may rename $L(u)$ for $u \in V(H)$ to obtain a $k$-matching assignment $\mathscr{M}'$ for $G$ such that all edges of $H$ are straight in $\mathscr{M}'$. 
\end{lemma}

A $d$-vertex, $d^{+}$-vertex or $d^{-}$-vertex is a vertex of degree $d$, at least $d$ or at most $d$ respectively. Similar definitions can be applied to faces and cycles. Liu and Li \cite{MR3881665} showed that every planar graph without a $4$-cycle adjacent to two triangles is DP-$4$-colorable. Actually, they showed the following stronger result. 

\begin{theorem}[Liu and Li \cite{MR3881665}]\label{LL}
Let $G$ be a planar graph without a $4$-cycle adjacent to two triangles, and let $\mathscr{M}$ be a $4$-matching assignment for $G$. If $S$ consists of exactly one vertex or all the vertices on a $6^{-}$-face of $G$, then every $\mathscr{M}$-coloring $\phi$ of $G[S]$ can be extended to an $\mathscr{M}$-coloring $\varphi$ of $G$. 
\end{theorem}

The first result of this paper is the following theorem which generalizes \autoref{LL}. 
\begin{restatable}{theorem}{MRA}\label{MRA}
Let $G$ be a plane graph without subgraphs isomorphic to the configurations in \autoref{F-MR1}, and let $\mathscr{M}$ be a $4$-matching assignment for $G$. If $S$ consists of exactly one vertex or all the vertices on a $6^{-}$-cycle of $G$, then every $\mathscr{M}$-coloring $\phi$ of $G[S]$ can be extended to an $\mathscr{M}$-coloring $\varphi$ of $G$. 
\end{restatable}

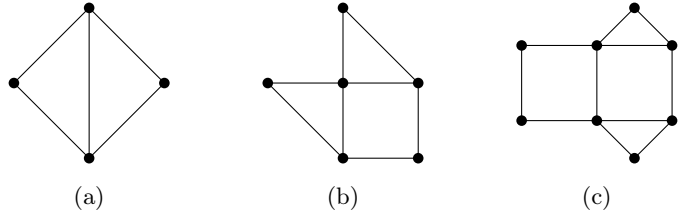
\begin{figure}[htbp]%
\centering
\subcaptionbox{\label{fig:subfig:a}}{\begin{tikzpicture}
\coordinate (A) at (1, 0);
\coordinate (B) at (0,1);
\coordinate (C) at (-1,0);
\coordinate (D) at (0,-1);
\draw (A)--(B)--(C)--(D)--cycle;
\draw (B)--(D);
\fill (A) circle (2pt)
(B) circle (2pt)
(C) circle (2pt)
(D) circle (2pt);
\end{tikzpicture}}
\hspace{1cm}
\subcaptionbox{\label{fig:subfig:b}}{\begin{tikzpicture}
\coordinate (A) at (1, 0);
\coordinate (B) at (0,1);
\coordinate (C) at (-1,0);
\coordinate (D) at (0,-1);
\coordinate (E) at (1, -1);
\draw (A)--(B)--(D)--(C)--(A)--(E)--(D);
\fill (0,0) circle (2pt)
(A) circle (2pt)
(B) circle (2pt)
(C) circle (2pt)
(D) circle (2pt)
(E) circle (2pt);
\end{tikzpicture}}
\hspace{1cm}
\subcaptionbox{\label{fig:subfig:c}}{\begin{tikzpicture}
\coordinate (A) at (1, 1);
\coordinate (B) at (0.5,1.5);
\coordinate (C) at (0,1);
\coordinate (D) at (-1,1);
\coordinate (E) at (-1, 0);
\coordinate (F) at (0, 0);
\coordinate (G) at (0.5, -0.5);
\coordinate (H) at (1, 0);
\draw (A)--(B)--(C)--(D)--(E)--(F)--(G)--(H)--cycle;
\draw (A)--(C)--(F)--(H);
\fill
(A) circle (2pt)
(B) circle (2pt)
(C) circle (2pt)
(D) circle (2pt)
(E) circle (2pt)
(F) circle (2pt)
(G) circle (2pt)
(H) circle (2pt);
\end{tikzpicture}}
\caption{The forbidden configurations in \autoref{MRA}.}
\label{F-MR1}
\end{figure}

In the above theorem, adjacent triangles are forbidden in the graph. In the second result of this paper, adjacent triangles are allowed but two triangles sharing exactly one vertex are forbidden. 
\begin{restatable}{theorem}{MRB}\label{MRB}
Let $G$ be a plane graph without subgraphs isomorphic to the configurations in \autoref{F-MR2}, and let $\mathscr{M}$ be a $4$-matching assignment for $G$. If $S$ is the vertices of a $6^{-}$-cycle of $G$, then every $\mathscr{M}$-coloring $\phi$ of $G[S]$ can be extended to an $\mathscr{M}$-coloring $\varphi$ of $G$. 
\end{restatable}

\begin{figure}[htbp]%
\centering
\subcaptionbox{\label{fig:subfig:2-a}}{\begin{tikzpicture}
\coordinate (A) at (1, 0);
\coordinate (B) at (0,-1);
\coordinate (C) at (-1,0);
\coordinate (D) at (0,1);
\coordinate (O) at (0, 0);
\draw (A)--(B)--(D)--(C)--cycle;
\fill
(A) circle (2pt)
(B) circle (2pt)
(C) circle (2pt)
(D) circle (2pt)
(O) circle (2pt);
\end{tikzpicture}}
\hspace{2cm}
\subcaptionbox{\label{fig:subfig:2-b}}{\begin{tikzpicture}
\coordinate (O) at (0, 0);
\coordinate (A) at (0, -1);
\coordinate (B) at (-1,-1);
\coordinate (C) at (-1,0);
\coordinate (D) at (135:1);
\coordinate (E) at (0,1);
\draw (O)--(A)--(B)--(C)--(D)--(E)--cycle;
\draw (C)--(O)--(D);
\fill
(O) circle (2pt)
(A) circle (2pt)
(B) circle (2pt)
(C) circle (2pt)
(D) circle (2pt)
(E) circle (2pt);
\end{tikzpicture}}
\caption{The forbidden configurations in \autoref{MRB}.}
\label{F-MR2}
\end{figure}
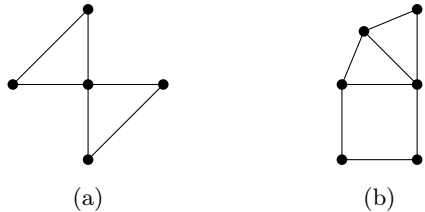

Wang and Lih \cite{MR1935837} showed that every planar graph without intersecting triangles is $4$-choosable. Luo \cite{MR2292964} improved it to toroidal graph, showing that every toroidal graph without intersecting triangles is $4$-choosable. As a corollary of \autoref{MRB}, the following corollary generalizes Wang and Lih's result. 

\begin{corollary}
Every planar graph without intersecting triangles is DP-$4$-colorable. 
\end{corollary}

Li and Wang showed the following result as a corollary in \cite{Li2019}. 
\begin{theorem}\label{House-}
Every toroidal graph without subgraphs isomorphic to the configurations in \autoref{NOA} is DP-$4$-colorable. 
\end{theorem}

\begin{figure}[htbp]%
\centering
\subcaptionbox{\label{fig:subfig:a-}}[0.25\linewidth]{\begin{tikzpicture}
\coordinate (A) at (45:1);
\coordinate (B) at (135:1);
\coordinate (C) at (225:1);
\coordinate (D) at (-45:1);
\coordinate (H) at (90:1.414);
\draw (A)--(H)--(B)--(C)--(D)--cycle;
\draw (A)--(B);
\fill (A) circle (2pt)
(B) circle (2pt)
(C) circle (2pt)
(D) circle (2pt)
(H) circle (2pt);
\end{tikzpicture}}
\subcaptionbox{\label{fig:subfig:b-}}[0.25\linewidth]{\begin{tikzpicture}
\coordinate (A) at (30:1);
\coordinate (B) at (150:1);
\coordinate (C) at (225:1);
\coordinate (D) at (-45:1);
\coordinate (H) at (90:1.414);
\coordinate (X) at (60:1.4);
\coordinate (Y) at (120:1.4);
\draw (A)--(X)--(H)--(Y)--(B)--(C)--(D)--cycle;
\draw (A)--(H)--(B);
\fill (A) circle (2pt)
(B) circle (2pt)
(C) circle (2pt)
(D) circle (2pt)
(H) circle (2pt)
(X) circle (2pt)
(Y) circle (2pt);
\end{tikzpicture}}
\caption{The forbidden configurations in \autoref{House-} and \ref{MRC}.}
\label{NOA}
\end{figure}
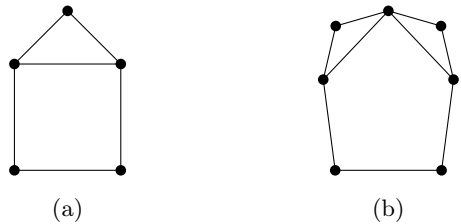

For planar graphs, we give the following result which is stronger than \autoref{House-}. 
\begin{restatable}{theorem}{MRC}\label{MRC}
Let $G$ be a plane graph without subgraphs isomorphic to the configurations in \autoref{NOA}, and let $\mathscr{M}$ be a $4$-matching assignment for $G$. If $S$ is the vertices of a $7^{-}$-cycle of $G$, then every $\mathscr{M}$-coloring $\phi$ of $G[S]$ can be extended to an $\mathscr{M}$-coloring $\varphi$ of $G$. 
\end{restatable}

We need some definitions for plane graphs in the following section. The unbounded face is called the {\bf outer} face, and the other faces are called {\bf inner faces}. An {\bf internal vertex} is a vertex that is not incident with outer face. An {\bf internal face} is a face having no common vertices with the outer cycle. Let $\mathcal{O}$ be a cycle of a plane graph $G$, the cycle $\mathcal{O}$ divides the plane into two regions, the subgraph induced by all the vertices in the unbounded region is denoted by $\ext(\mathcal{O})$, and the subgraph induced by all the vertices in the other region is denoted by $\int(\mathcal{O})$. If both $\int(\mathcal{O})$ and $\ext(\mathcal{O})$ contain at least one vertex, then we call the cycle $\mathcal{O}$ a {\bf separating cycle} of $G$. The subgraph induced by the complement of $V(\ext(\mathcal{O}))$ is denoted by $\Int(\mathcal{O})$, and the subgraph induced by the complement of $V(\int(\mathcal{O}))$ is denoted by $\Ext(\mathcal{O})$. 

Let $G$ be a plane graph and the outer face be bounded by a cycle. A $4$-vertex is a {\bf $4_{k}$-vertex} if it is incident with exactly $k$ triangular-faces. An inner $3$-face is a {\bf $\mathcal{T}_{k}$-face} if it has exactly $k$ common vertices with the outer cycle. A $3$-face is a {\bf special face} if it is an internal face incident with a $4_{1}$-vertex. A {\bf special edge} is an edge having a common vertex with the outer cycle but it is not an edge of the outer cycle. Let $\mathcal{N}$ be the set of inner faces having at least one common vertex with the outer face. 

\section{Proof of \autoref{MRA}}\label{sec2}
\MRA*
\begin{proof}
Suppose that $G$ is a minimal counterexample to \autoref{MRA}. That is, there exists an $\mathscr{M}$-coloring of $G[S]$ that can not be extended to an $\mathscr{M}$-coloring of $G$ such that 
\begin{equation}\label{EQ1-MR1}
|V(G)| \text{ is minimized.}
\end{equation}
Subject to \eqref{EQ1-MR1}, 
\begin{equation}\label{EQ2-MR1}
|E(G)| - |S| \text{ is minimized.}
\end{equation}

The following structural results \autoref{L1-MR1}\ref{La-MR1}--\ref{Lf-MR1} are almost the same with that in \cite{MR3881665, Lu}, so the proofs are omitted here. Let $f$ be an internal $(4, 4, 4, 4, 4^{+})$-face adjacent to five 3-faces. If $w$ is incident with one of the five 3-faces but not on $f$, then we call $w$ a related {\bf source} of $f$ and $f$ a {\bf sink} of $w$. 

\begin{lemma}\label{L1-MR1}
\text{}
\begin{enumerate}[label = (\alph*)]
\item\label{La-MR1} $S \neq V(G)$; 
\item\label{Lb-MR1} $G$ is $2$-connected, and thus the boundary of every face is a cycle; 
\item\label{Lc-MR1} each vertex not in $S$ has degree at least four; 
\item\label{Ld-MR1} either $|S| = 1$ or $G[S]$ is an induced cycle of $G$; 
\item\label{Le-MR1} there is no separating $k$-cycle for $3 \leq k \leq 6$;
\item\label{Lf-MR1} $G[S]$ is an induced cycle of $G$, so we may assume that $\mathcal{C} = G[S]$ is the outer cycle; 
\item\label{Lg-MR1} if $x$ and $y$ are two nonconsecutive vertices on $\mathcal{C}$, then they have no common neighbor in $\int(\mathcal{C})$;
\item\label{Lh-MR1} if $f$ is a sink in $G$, then at most one of its source is on the outer cycle $\mathcal{C}$;  
\item\label{Li-MR1} there is no $4_{2}$-vertex incident with a $4$-face. 
\qed
\end{enumerate}
\end{lemma}
\begin{proof}
\ref{Lg-MR1} Suppose that $x$ and $y$ are two nonconsecutive vertices on $\mathcal{C}$, and they have a common neighbor $z \in \int(\mathcal{C})$. The two vertices $x$ and $y$ divide the cycle $\mathcal{C}$ into two paths $P_{1}[x, y]$ and $P_{2}[x, y]$. It is observed that $P_{1}[x, y] \cup xzy$ and $P_{2}[x, y] \cup xzy$ are two cycles with length at most $|\mathcal{C}|$, thus these two cycles are not separating cycles by \autoref{L1-MR1}\ref{Le-MR1}, and thus $z$ must be an internal $2$-vertex, but this contradicts \autoref{L1-MR1}\ref{Lc-MR1}. 

\ref{Lh-MR1} Let $f = v_{1}v_{2}v_{3}v_{4}v_{5}$ be an internal 5-face adjacent to five 3-face $x_{i}v_{i}v_{i+1}$ for $1 \leq i \leq 5$. It is observed that $x_{1}, x_{2}, x_{3}, x_{4}$ and $x_{5}$ are five distinct vertices. Suppose that $x_{i}$ and $x_{i+1}$ are on the outer cycle $\mathcal{C}$, where the subscripts are taken module~$5$. By \autoref{L1-MR1}\ref{Lg-MR1}, $x_{i}$ and $x_{i+1}$ are consecutive on $\mathcal{C}$, thus $x_{i}v_{i+1}x_{i+1}$ is a triangle, but this contradicts that there is no adjacent triangles. 

Suppose that $x_{i}$ and $x_{i+2}$ are on the outer cycle $\mathcal{C}$, where the subscripts are taken module~$5$. By the above arguments, $x_{i+1}$ and $x_{i+3}$ are internal vertices. The outer cycle $\mathcal{C}$ and the path $P = x_{i}v_{i+1}v_{i+2}x_{i+2}$ form two cycles containing $P$. It is observed that these two cycles are separating cycles, thus the lengths are at least 7 by \autoref{L1-MR1}\ref{Le-MR1}, and thus $|\mathcal{C}| \geq 2 \times 7 - 2|P| = 8$, a contradiction. 

\ref{Li-MR1} If a $4$-face is incident with a $4_{2}$-vertex, then \autoref{fig:subfig:a} or \autoref{fig:subfig:b} will appear in $G$, a contradiction. 
\end{proof}

The next two structural results focus on the internal $4$-vertices. 
\begin{lemma}\label{FOUR1-MR1}
Let $w$ be a $4$-vertex with four neighbors $w_{1}, w_{2}, w_{3}, w_{4}$ in a cyclic order. If $w_{1}, w_{3}$ and $w$ are all internal vertices, then at most one of $w_{1}$ and $w_{3}$ is a $4$-vertex. 
\end{lemma}
\begin{proof}
Suppose to the contrary that both $w_{1}$ and $w_{3}$ are $4$-vertices. By \autoref{ST}, we may assume that all the four edges incident with $w$ are straight. Let $\Gamma \coloneqq G - \{w, w_{1}, w_{3}\}$, let $G'$ be obtained from $\Gamma$ by identifying $w_{2}$ and $w_{4}$, and let $\mathscr{M}'$ be the restriction of $\mathscr{M}$ on $E(G')$. Let $P$ be a shortest path from $w_{2}$ to $w_{4}$ in $\Gamma$. If the length of $P$ is at most four, then $P \cup w_{2}ww_{4}$ is a separating cycle with length at most six, a contradiction. Therefore, the distance of $w_{2}$ and $w_{4}$ in $\Gamma$ is at least five. Note that the precolored cycle has length at most six and the length of $P$ is at least five in $\Gamma$, thus $G'[S] = G[S] = \mathcal{C}$ and $\phi$ is also an $\mathscr{M}'$-coloring of $G'[S]$. Moreover, $G'$ is a simple plane graph without subgraphs as in \autoref{F-MR1}. By the minimality, $\phi$ can be extended to an $\mathscr{M}'$-coloring $\phi'$ of $G'$. There are at most three forbidden colors for each of $w_{1}$ and $w_{3}$, so we can extend the coloring $\phi'$ to $w_{1}$ and $w_{3}$. Since $w_{2}$ and $w_{4}$ have the same color and $ww_{2}, ww_{4}$ are straight, we can further extend the coloring to $w$, a contradiction. 
\end{proof}

\begin{lemma}\label{FOUR2-MR1}
Let $w$ be an internal $4_{1}$-vertex with four neighbors $w_{1}, w_{2}, w_{3}, w_{4}$ in a cyclic order. If $ww_{1}$ is incident with a $3$-face, then $w_{1}$ cannot be an internal $4$-vertex. 
\end{lemma}

\begin{proof}
Suppose to the contrary that $w_{1}$ is an internal $4$-vertex. By \autoref{ST}, we may assume that all the four edges incident with $w$ are straight. Let $\Gamma \coloneqq G - \{w, w_{1}\}$, let $G'$ be obtained from $\Gamma$ by identifying $w_{2}$ and $w_{4}$, and $\mathscr{M}'$ be the restriction of $\mathscr{M}$ on $E(G')$.  Suppose that $P$ is a path from $w_{2}$ to $w_{4}$ in $\Gamma$ with length at most three. Thus, $\mathcal{Q} = P \cup w_{2}ww_{4}$ is a cycle of length at most five. If $w_{3}$ is not on $P$, then $\mathcal{Q}$ is a separating $5^{-}$-cycle, a contradiction. Otherwise $w_{3}$ is on $P$, thus \autoref{L1-MR1}\ref{Le-MR1} implies that $ww_{3}$ is incident with a $3$-face and another $4^{-}$-face, but this contradicts the fact that the configurations in \autoref{fig:subfig:a} and \autoref{fig:subfig:b} are forbidden in $G$, a contradiction. Therefore, the distance of $w_{2}$ and $w_{4}$ in $\Gamma$ is at least four, and at least one of $w_{2}$ and $w_{4}$ is an internal vertex. 

\begin{figure}[htbp]%
\centering
\def\s{0.8}
\begin{tikzpicture}
\node [circle, draw] (c) at (0, 0) [minimum size = 3cm*\s]{};
\coordinate (A) at (0, 1.5*\s);
\coordinate (B) at (0, -1.5*\s);
\coordinate (w4) at (0, 0);
\coordinate (w) at (0, -0.75*\s);
\coordinate (w1) at (-0.75*\s, -0.75*\s);
\coordinate (w3) at (0.75*\s, -0.75*\s); 
\draw (A) node[above]{\small$u$}--(w4) node[right]{\small$w_{4}$}--(w) node[below right]{\small$w$}--(B) node[below]{\small$w_{2}$};
\draw (w1) node[above]{$w_{1}$}--(w)--(w3);
\fill 
(A) circle (2pt)
(B) circle (2pt)
(w) circle (2pt)
(w1) circle (2pt)
(w4) circle (2pt);
\node [left] at (-1.5*\s, 0) {\small $P_{1}$}; 
\node [right] at (1.5*\s, 0) {\small $P_{2}$}; 
\end{tikzpicture}
\caption{Local structure.}
\label{C-MR1}
\end{figure}
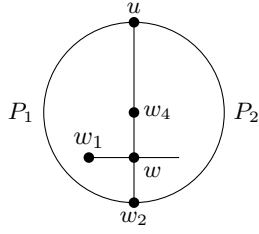

Suppose that the identification creates a chord $uw^{*}$ of $\mathcal{C}$, where $w^{*}$ is the new vertex. The outer cycle $\mathcal{C}$ is divided into two paths $P_{1}$ and $P_{2}$ by $u$ and $w_{2}$, see \autoref{C-MR1}. Recall that the distance of $w_{2}$ and $w_{4}$ is at least four in $G - \{w, w_{1}\}$, thus $P_{1}$ and $P_{2}$ have lengths at least three. Since $\mathcal{C}$ is a cycle of length at most six, $P_{1}$ and $P_{2}$ have lengths exactly three, which implies that $P_{1} \cup uw_{4}ww_{2}$ is a separating cycle of length six, a contradiction. Therefore, the identification does not create any chord of $\mathcal{C}$, and $\phi$ is an $\mathscr{M}'$-coloring of $G'[S]$. Since the distance of $w_{2}$ and $w_{4}$ in $\Gamma$ is at least four, $G'$ has no loops and no multiple edges. It is easy to check that there are no subgraphs isomorphic to the configurations in \autoref{F-MR1}. By the minimality, $\phi$ can be extended to an $\mathscr{M}'$-coloring $\phi'$ of $G'$. There are at most three forbidden colors for $w_{1}$, so we can extend the coloring $\phi'$ to $w_{1}$. Since $w_{2}$ and $w_{4}$ have the same color and $ww_{2}, ww_{4}$ are straight, we can further extend the coloring to $w$, a contradiction. 
\end{proof}

We give the initial charge $\mu(v) = 2\deg(v) - 6$ for any $v \in V(G)$, $\mu(D) = \deg(D) + 6$ for the outer face $D$, and $\mu(f) = \deg(f) - 6$ for any face $f$ other than $D$. By Euler's formula, the sum of the initial charges is zero. That is, 
\begin{equation*}\label{Euler2-MR1}
\sum_{v\,\in\,V(G)}\big(2\deg(v) - 6\big) + \sum_{f\,\in\,F(G) \setminus D}\big(\deg(f) - 6\big) + \big(\deg(D) + 6\big)= 0. 
\end{equation*}

We design discharging rules to redistribute the charges, preserving the sum, such that every element in $V(G) \cup F(G)$ has a nonnegative final charge $\mu'$. Moreover, there exists a face $f^{*}$ having a positive final charge, which leads to a contradiction. 

\begin{enumerate}[label = \bf R\arabic*., ref = R\arabic*]
\item\label{R1-MR1} Let $v$ be an internal $4$-vertex and $t$ be the number of incident $3$-faces. 
\begin{enumerate}[label = \bf \alph*., ref = \alph*]
\item\label{R1a-MR1} If $t = 2$, then $v$ sends $1$ to each incident $3$-face. 
\item\label{R1b-MR1} If $t = 1$ and it is incident with a $\mathcal{T}_{2}$-face $g$, then $v$ sends $1$ to the incident $\mathcal{T}_{2}$-face, sends $\frac{1}{2}$ to the incident face not adjacent to $g$, and sends $\frac{1}{4}$ to each incident $4^{+}$-face that shares an edge with $g$. 
\item\label{R1c-MR1} If $t = 1$ and it is not incident with a $\mathcal{T}_{2}$-face, then $v$ sends $\frac{1}{2}$ to each incident face. 
\item\label{R1d-MR1} If $t = 0$, then $v$ sends $\frac{1}{2}$ to each incident face. 
\end{enumerate}

\item\label{R2-MR1} Every internal $5^{+}$-vertex sends $\frac{5}{4}$ to each incident special face, sends $1$ to each of the other incident $3$-face, and sends $\frac{1}{2}$ to each incident $4^{+}$-face. 
\item\label{R3-MR1} Every vertex on the outer cycle $\mathcal{C}$ sends its initial charge to the outer face $D$, and $D$ sends $2$ to each special edge, and each special edge immediately sends $1$ to each incident face. 
\item\label{Source} Every internal $5^{+}$-vertex $z$ sends $\frac{1}{4}$ via incident 3-face $xyz$ to the sink. 
\end{enumerate}
\begin{remark}\label{Remark-1}
Note that each special edge receives $2$ from the outer face and sends $1$ to each incident face, thus each special edge has final charge zero and each face in $\mathcal{N}$ receives at least $2$ from the outer face via special edges. 
\end{remark}

Each vertex on $\mathcal{C}$ has final charge zero by \ref{R3-MR1}, thus it suffices to consider all the faces and internal vertices. By \autoref{L1-MR1}\ref{Lc-MR1}, every internal vertex has degree at least four. Since there is no adjacent $3$-faces, every $k$-vertex is incident with at most $\lfloor\frac{k}{2}\rfloor$ triangular-faces. In particular, every $4$-vertex is incident with at most two $3$-faces. 

Let $v$ be an internal $4$-vertex. If $v$ is an internal $4_{2}$-vertex, then $\mu'(v) = 2 - 2 \times 1 = 0$ by \ref{R1-MR1}\ref{R1a-MR1}. If $v$ is an internal $4_{1}$-vertex and it is incident with a $\mathcal{T}_{2}$-face, then $\mu'(v) \geq 2 - 1 - \frac{1}{2} - 2 \times \frac{1}{4} = 0$ by \ref{R1-MR1}\ref{R1b-MR1}. If $v$ is an internal $4_{1}$-vertex but it is not incident with any $\mathcal{T}_{2}$-face, then $\mu'(v) = 2 - 4 \times \frac{1}{2} = 0$ by \ref{R1-MR1}\ref{R1c-MR1}. If $v$ is an internal $4_{0}$-vertex, then $\mu'(v) = 2 - 4 \times \frac{1}{2} = 0$ by \ref{R1-MR1}\ref{R1d-MR1}. 

Let $v$ be an internal $5^{+}$-vertex. Every internal $5^{+}$-vertex sends at most $\frac{5}{4}$ to/via each incident $3$-face, and sends $\frac{1}{2}$ to each incident $4^{+}$-face by \ref{R2-MR1} and \ref{Source}. Thus,  
\begin{equation*}
\mu'(v) \geq 2\deg(v) - 6 - \left\lfloor \frac{\deg(v)}{2} \right\rfloor \times \frac{5}{4} - \left\lceil \frac{\deg(v)}{2} \right\rceil \times \frac{1}{2} \geq 0. 
\end{equation*}

Let $f$ be an inner $3$-face. It is observed that there is no $\mathcal{T}_{3}$-faces. If $f$ is a $\mathcal{T}_{2}$-face, then it receives $1$ from the incident internal vertex, and then $\mu'(f) = -3 + 1 + 2 = 0$ by \ref{R1-MR1}, \ref{R2-MR1} and \autoref{Remark-1}. If $f$ is a $\mathcal{T}_{1}$-face, then it receives at least $\frac{1}{2}$ from each incident internal vertex, and then $\mu'(f) \geq -3 + 2 \times \frac{1}{2} + 2 = 0$ by \ref{R1-MR1}, \ref{R2-MR1} and \autoref{Remark-1}. If $f$ is an internal $3$-face but not a special face, then $\mu'(f) = -3 + 3 \times 1 = 0$ by \ref{R1-MR1}\ref{R1a-MR1} and \ref{R2-MR1}. Assume that $f$ is a special face. By \autoref{FOUR2-MR1}, $f$ is a $(4, 5^{+}, 5^{+})$-face. Thus, $f$ receives $\frac{1}{2}$ from the incident $4$-vertex and $\frac{5}{4}$ from each incident $5^{+}$-vertex by \ref{R1-MR1}\ref{R1c-MR1} and \ref{R2-MR1}, which implies that $\mu'(f) = -3 + \frac{1}{2} + 2 \times \frac{5}{4} = 0$. 
 
Let $f$ be an inner $4^{+}$-face. If $f$ is a $4^{+}$-face in $\mathcal{N}$, then $\mu'(f) \geq -2 + 2 = 0$ by \autoref{Remark-1}. If $f$ is an internal $4$-face, then no vertex on $f$ can be a $4_{2}$-vertex by \autoref{L1-MR1}\ref{Li-MR1}, which implies that $\mu'(f) = -2 + 4 \times \frac{1}{2} = 0$. If $f$ is an internal $6^{+}$-face, then $\mu'(f) \geq \mu(f) =  \deg(f) - 6 \geq 0$. So we may assume that $f$ is an internal $5$-face in the remaining of this paragraph. If $f$ is incident with at least two $5^{+}$-vertices, then $\mu'(f) \geq -1 + 2 \times \frac{1}{2} = 0$ by \ref{R2-MR1}. So we may further assume that $f$ is a $(4, 4, 4, 4, 4^{+})$-face. If $f$ is adjacent to a $4^{+}$-face via $xy$, then neither $x$ nor $y$ is a $4_{2}$-vertex, and then each of $x$ and $y$ sends $\frac{1}{2}$ to $f$, which implies that $\mu'(f) \geq -1 + 2 \times \frac{1}{2} = 0$. Hence, $f$ is an internal $(4, 4, 4, 4, 4^{+})$-face adjacent to five 3-faces. By \autoref{L1-MR1}\ref{Lh-MR1}, at least four sources are internal vertices. By \autoref{FOUR1-MR1}, every internal source is a $5^{+}$-vertex, and it sends $\frac{1}{4}$ to the sink $f$ by \ref{Source}, which implies that $\mu'(f) \geq -1 + 4 \times \frac{1}{4} = 0$. 

By \ref{R3-MR1}, the final charge of the outer face $D$ is 
\begin{equation*}
\mu'(D) = \deg(D) + 6 + \sum_{v \in V(D)} \big(2\deg(v) - 6\big) - 2\left(\sum_{v \in V(D)} \big(\deg(v) - 2\big)\right) = 6 - \deg(D) \geq 0. 
\end{equation*}

Since $G$ is $2$-connected and $S \neq V(G)$, there are at least two special edges. Since there is no adjacent $3$-faces, there exists a $4^{+}$-face $f^{*}$ incident with a special edge. If $f^{*}$ is a $5^{+}$-face, then $\mu'(f^{*}) \geq -1 + 2 > 0$ by \autoref{Remark-1}. So we may assume that $f^{*}$ is a $4$-face incident with an internal vertex $v$. Since $f^{*}$ is in $\mathcal{N}$, it receives $2$ from $D$. If $v$ is a $5^{+}$-vertex, then it sends $\frac{1}{2}$ to $f^{*}$ by \ref{R2-MR1}; while $v$ is a $4$-vertex, it is a $4_{0}$-vertex or $4_{1}$-vertex, so it sends at least $\frac{1}{4}$ to $f^{*}$ by \ref{R1-MR1}. In these two cases, $\mu'(f^{*}) \geq  - 2 + 2 + \frac{1}{4} > 0$. Therefore, there must exist a $4^{+}$-face $f^{*}$ incident with a special edge such that its final charge is positive. This completes the proof. 
\end{proof}

\section{Proof of \autoref{MRB}}
\MRB*
\begin{proof}
Suppose that $G$ is a minimal counterexample to \autoref{MRB}. That is, there exists an $\mathscr{M}$-coloring of $G[S]$ that can not be extended to an $\mathscr{M}$-coloring of $G$ such that 
\begin{equation}\label{EQ1-MR2}
|V(G)| \text{ is minimized.}
\end{equation}
Subject to \eqref{EQ1-MR2}, 
\begin{equation}\label{EQ2-MR2}
|E(G)| - |S| \text{ is minimized.}
\end{equation}

The following properties are very similar to \autoref{L1-MR1}, so we omit the proofs here. 
\begin{lemma}\label{L1-MR2}
\text{}
\begin{enumerate}[label = (\alph*)]
\item\label{La-MR2} $S \neq V(G)$ and $G$ is connected; 
\item\label{Lb-MR2} $\mathcal{C} = G[S]$ is an induced cycle of $G$; 
\item\label{Lc-MR2} there is no separating $k$-cycle for $3 \leq k \leq 6$, so we may assume that $\mathcal{C}$ is the outer cycle; 
\item\label{Ld-MR2} each internal vertex has degree at least four.
\item\label{Le-MR2} every $4^{+}$-vertex is incident with at most two $3$-faces;
\item\label{Lf-MR2} there is no $4_{2}$-vertex incident with a $4$-face. 
\qed
\end{enumerate}
\end{lemma}

The following result and its proof are almost the same with \autoref{FOUR2-MR1}, so we only present the result but without giving the proof. Note that $w$ is incident with exactly one $3$-face in \autoref{FOUR2-MR2}. 

\begin{lemma}\label{FOUR2-MR2}
Let $w$ be an internal $4_{1}$-vertex with four neighbors $w_{1}, w_{2}, w_{3}, w_{4}$ in a cyclic order. If $ww_{1}$ is incident with a $3$-face, then $w_{1}$ cannot be an internal $4$-vertex. 
\end{lemma}

We give the initial charge $\mu(v) = 2\deg(v) - 6$ for any $v \in V(G)$, $\mu(D) = \deg(D) + 6$ for the outer face $D$, and $\mu(f) = \deg(f) - 6$ for any face $f$ other than $D$. By Euler's formula, the sum of the initial charges is zero. That is, 
\begin{equation*}\label{Euler2-MR2}
\sum_{v\,\in\,V(G)}\big(2\deg(v) - 6\big) + \sum_{f\,\in\,F(G) \setminus D}\big(\deg(f) - 6\big) + \big(\deg(D) + 6\big)= 0. 
\end{equation*}

Next, we give the discharging rules to redistribute the charges, preserving the sum, such that every element in $V(G) \cup F(G)$ has a nonnegative final charge $\mu'$. Moreover, there exists a face $f^{*}$ having a positive final charge, which leads to a contradiction. 

\begin{enumerate}[label = \bf R\arabic*., ref = R\arabic*]
\item\label{R1-MR2} Let $v$ be an internal $4$-vertex and $t$ be the number of incident $3$-faces. 
\begin{enumerate}[label = \bf \alph*., ref = \alph*]
\item\label{R1a-MR2} If $t = 2$, then $v$ sends $1$ to each incident $3$-face. 
\item\label{R1b-MR2} If $t \leq 1$, then $v$ sends $1$ to each incident $\mathcal{T}_{2}$-face, and sends $\frac{1}{2}$ to each incident internal $4$-face, internal $5$-face and each of the other incident $3$-face, and sends $\frac{1}{4}$ to each incident $4^{+}$-face in $\mathcal{N}$. 
\end{enumerate}
\item\label{R2-MR2} Every internal $5^{+}$-vertex sends $\frac{5}{4}$ to each incident $3$-face, and sends $\frac{1}{2}$ to each incident $4^{+}$-face. 
\item\label{R3-MR2} Each vertex on the outer cycle $\mathcal{C}$ sends its initial charge to the outer face $D$. The outer face $D$ sends $2$ to each special edge, and each special edge immediately sends $1$ to each incident face.  
\end{enumerate}
\begin{remark}\label{Remark-2}
Note that each special edge receives $2$ from the outer face and sends $1$ to each incident face, thus each special edge has final charge zero and each face in $\mathcal{N}$ receives at least $2$ from the outer face via special edges. 
\end{remark}

Each vertex on $\mathcal{C}$ has final charge zero by \ref{R3-MR2}, thus it suffices to consider all the faces and internal vertices. By \autoref{L1-MR2}\ref{Ld-MR2}, every internal vertex has degree at least four. By \autoref{L1-MR2}\ref{Le-MR2}, every $4^{+}$-vertex is incident with at most two $3$-faces. 

Let $v$ be an internal $4$-vertex. If $v$ is an internal $4_{2}$-vertex, then $\mu'(v) = 2 - 2 \times 1 = 0$ by \ref{R1-MR2}\ref{R1a-MR2}. If $v$ is an internal $4_{1}$-vertex and it is incident with a $\mathcal{T}_{2}$-face, then $\mu'(v) \geq 2 - 1 - \frac{1}{2} - 2 \times \frac{1}{4} = 0$ by \ref{R1-MR2}\ref{R1b-MR2}. If $v$ is an internal $4_{1}$-vertex but it is not incident with any $\mathcal{T}_{2}$-face, then $\mu'(v)\geq 2 - 4 \times \frac{1}{2} = 0$ by \ref{R1-MR2}\ref{R1b-MR2}. If $v$ is an internal $4_{0}$-vertex, then $\mu'(v) \geq 2 - 4 \times \frac{1}{2} = 0$ by \ref{R1-MR2}\ref{R1b-MR2}. 

Let $v$ be an internal $5^{+}$-vertex. By \ref{R2-MR2}, $v$ sends $\frac{5}{4}$ to each incident $3$-face, and sends $\frac{1}{2}$ to each incident $4^{+}$-face. Thus, $\mu'(v) \geq 2\deg(v) - 6 - 2 \times \frac{5}{4} - \big(\deg(v) - 2\big) \times \frac{1}{2} = \frac{3}{2}\big(\deg(v) - 5\big) \geq 0$. 

Let $f$ be an inner $3$-face. It is observed that there is no $\mathcal{T}_{3}$-faces. If $f$ is a $\mathcal{T}_{2}$-face, then it receives $1$ from the incident internal vertex, and then $\mu'(f) \geq -3 + 1 + 2 = 0$ by \ref{R1-MR2}, \ref{R2-MR2} and \autoref{Remark-2}. If $f$ is a $\mathcal{T}_{1}$-face, then it receives at least $\frac{1}{2}$ from each incident internal vertex, and then $\mu'(f) \geq -3 + 2 \times \frac{1}{2} + 2 = 0$ by \ref{R1-MR2}, \ref{R2-MR2} and \autoref{Remark-2}. If $f$ is an internal $3$-face but not a special face, then $\mu'(f) = -3 + 3 \times 1 = 0$ by \ref{R1-MR2}\ref{R1a-MR2} and \ref{R2-MR2}. Assume that $f$ is a special face. By \autoref{FOUR2-MR2}, $f$ is a $(4, 5^{+}, 5^{+})$-face. Thus, $f$ receives $\frac{1}{2}$ from the incident $4$-vertex and $\frac{5}{4}$ from each incident $5^{+}$-vertex by \ref{R1-MR2}\ref{R1b-MR2} and \ref{R2-MR2}, which implies that $\mu'(f) = -3 + \frac{1}{2} + 2 \times \frac{5}{4} = 0$. 

Let $f$ be an inner $4^{+}$-face. If $f$ is a $4^{+}$-face in $\mathcal{N}$, then $\mu'(f) \geq -2 + 2 = 0$ by \autoref{Remark-2}. If $f$ is an internal $4$-face, then no vertex on $f$ can be a $4_{2}$-vertex by \autoref{L1-MR2}\ref{Lf-MR2}, which implies that $\mu'(f) = -2 + 4 \times \frac{1}{2} = 0$. If $f$ is an internal $6^{+}$-face, then $\mu'(f) \geq \mu(f) = \deg(f) - 6 \geq 0$. Let $f$ be an internal $5$-face. Since there is no subgraph isomorphic to the configuration in \autoref{fig:subfig:2-a}, $f$ is adjacent to at most two $3$-faces. If $f$ is adjacent to a $3$-face $uvw$ with $uv$ on $f$, then at most one of $u$ and $v$ is a $4_{2}$-vertex. Hence, $f$ is incident with at most two $4_{2}$-vertices. By the discharging rules, if $x$ is a vertex on $f$ but it is not a $4_{2}$-vertex, then it sends $\frac{1}{2}$ to $f$. This implies that $\mu'(f) \geq - 1 + 3 \times \frac{1}{2} > 0$.

By \ref{R3-MR2}, the final charge of the outer face $D$ is 
\begin{equation*}
\mu'(D) = \deg(D) + 6 + \sum_{v \in V(D)} \big(2\deg(v) - 6\big) - 2\left(\sum_{v \in V(D)} \big(\deg(v) - 2\big)\right) = 6 - \deg(D) \geq 0. 
\end{equation*}

\begin{lemma}\label{4+face}
There is a $4^{+}$-face incident with a special edge. 
\end{lemma}
\begin{proof}
Since $G$ is connected and $S \neq V(G)$, there is at least one special edge. Suppose that $xuy$ is a path on the outer cycle $\mathcal{C}$ and it is incident with a special edge $uv$. If every inner face incident with $u$ is a $3$-face, then $u$ must has degree three and it is incident with two $3$-faces, say $uvx$ and $uvy$, for otherwise there is a subgraph isomorphic to the configuration in \autoref{fig:subfig:2-a}. Since every $3$-cycle bounds a $3$-face and every internal vertex has degree at least four, $x$ and $y$ are nonadjacent. This implies that $vx$ is incident with a $4^{+}$-face, for otherwise there is a subgraph isomorphic to the configuration in \autoref{fig:subfig:2-a}. 
\end{proof}

By \autoref{4+face}, there is a $4^{+}$-face $f^{*}$ incident with a special edge $uv$ and $v$ is an internal vertex. If $f^{*}$ is a $5^{+}$-face, then $\mu'(f^{*}) \geq -1 + 2 > 0$ by \autoref{Remark-2}. So we may assume that $f^{*}$ is a $4$-face. By \autoref{L1-MR2}\ref{Lf-MR2}, there is no $4_{2}$-vertex incident with $4$-face $f^{*}$. If $v$ is a $4$-vertex, then it is incident with at most one $3$-face, and sends at least $\frac{1}{4}$ to $f^{*}$ by \ref{R1-MR2}\ref{R1b-MR2}, thus $\mu'(f^{*}) \geq -2 + 2 + \frac{1}{4} > 0$. If $v$ is a $5^{+}$-vertex, then $\mu'(f^{*}) \geq -2 + 2 + \frac{1}{2} > 0$ by \ref{R2-MR2} and \autoref{Remark-2}. This completes the proof. 
\end{proof}

\section{Proof of \autoref{MRC}}
\MRC*
\begin{proof}
Suppose that $G$ is a minimal counterexample to \autoref{MRC}. That is, there exists an $\mathscr{M}$-coloring of $G[S]$ that can not be extended to an $\mathscr{M}$-coloring of $G$ such that 
\begin{equation}\label{EQ1-MR3}
|V(G)| \text{ is minimized.}
\end{equation}
Subject to \eqref{EQ1-MR3}, 
\begin{equation}\label{EQ2-MR3}
|E(G)| - |S| \text{ is minimized.}
\end{equation}

Similarly, we only give the following structural results but without giving the proofs. 
\begin{lemma}\label{L1-MR3}
\text{}
\begin{enumerate}[label = (\alph*)]
\item\label{La-MR3} $S \neq V(G)$ and $G$ is connected; 
\item\label{Lb-MR3} $\mathcal{C} = G[S]$ is an induced cycle of $G$; 
\item\label{Lc-MR3} there is no separating $k$-cycle for $3 \leq k \leq 7$, so we may assume that $\mathcal{C}$ is the outer cycle of $G$; 
\item\label{Ld-MR3} each internal vertex has degree at least four.
\item\label{Le-MR3} if $u$ and $v$ on $\mathcal{C}$ are nonadjacent, then they have no common neighbor not on $\mathcal{C}$;
\item\label{Lf-MR3} every $4$-face is incident with four $4^{+}$-face. 
\qed
\end{enumerate}
\end{lemma}

We give the initial charge $\mu(v) = 2\deg(v) - 6$ for any $v \in V(G)$, $\mu(f) = \deg(f) - 6$ for any inner face $f \in F(G)$, and $\mu(D) = \deg(D) + 6$ for the outer face $D$. By Euler's formula, the sum of the initial charges is zero. That is, 
\begin{equation*}\label{Euler2-MR3}
\sum_{v\,\in\,V(G)}\big(2\deg(v) - 6\big) + \sum_{f\,\in\,F(G) \setminus D}\big(\deg(f) - 6\big) + \big(\deg(D) + 6\big)= 0. 
\end{equation*}
Next, we give the discharging rules to redistribute the charges, preserving the sum, such that every element in $V(G) \cup F(G)$ other than $D$ has a nonnegative charge $\mu'$ after applying \ref{R1-}, \ref{R2-} and \ref{R3-}. Moreover, the outer face $D$ has a positive final charge $\mu^{*}$ after applying \ref{TOD}, which leads to a contradiction. 

\begin{enumerate}[label = \bf R\arabic*., ref = R\arabic*]
\item\label{R1-} Let $v$ be an internal $4$-vertex and $t$ be the number of incident $3$-faces. 
\begin{enumerate}[label = \bf \alph*., ref = \alph*]
\item\label{R1a-} If $t = 2$, then $v$ sends $1$ to each incident $3$-face. 
\item\label{R1b-} If $t = 1$ and it is incident with a $3$-face $g$, then $v$ sends $1$ to the $3$-face $g$, and sends $\frac{1}{2}$ to the incident face not adjacent to $g$, and sends $\frac{1}{4}$ to each incident $4^{+}$-face that shares an edge with $g$. 
\item\label{R1c-} If $t = 0$, then $v$ sends $\frac{1}{2}$ to each incident face. 
\end{enumerate}
\item\label{R2-} Every internal $5^{+}$-vertex sends $1$ to each incident $3$-face, and sends $\frac{1}{2}$ to each incident $4^{+}$-face. 
\item\label{R3-} Each vertex on the outer cycle $\mathcal{C}$ sends its initial charge to the outer face. The outer face sends $2$ to each special edge, and each special edge immediately sends $1$ to each incident face. 
\item\label{TOD} Every inner face transfers its surplus charge to $D$. 
\end{enumerate}
\begin{remark}\label{Remark-3}
Note that each special edge receives $2$ from the outer face and sends $1$ to each incident face, thus each special edge has final charge zero and each face in $\mathcal{N}$ receives at least $2$ from the outer face via special edges. 
\end{remark}

Each vertex on $\mathcal{C}$ has final charge zero by \ref{R3-}, thus it suffices to consider all the faces and internal vertices. By \autoref{L1-MR3}\ref{Ld-MR3}, every internal vertex has degree at least four. 

By the absence of \autoref{fig:subfig:a-}, every $k$-vertex is incident with at most $\lfloor\frac{2k}{3}\rfloor$ triangular-faces. In particular, every $4$-vertex is incident with at most two $3$-faces. If $v$ is an internal $4_{2}$-vertex, then $\mu'(v) = 2 - 2 \times 1 = 0$ by \ref{R1-}\ref{R1a-}. If $v$ is an internal $4_{1}$-vertex, then $\mu'(v) \geq 2 - 1 - \frac{1}{2}  - 2 \times \frac{1}{4} = 0$ by \ref{R1-}\ref{R1b-}. If $v$ is an internal $4_{0}$-vertex, then $\mu'(v) = 2 - 4 \times \frac{1}{2} = 0$ by \ref{R1-}\ref{R1c-}. 

Let $v$ be an internal $5^{+}$-vertex. By \ref{R2-}, $v$ sends $1$ to each incident $3$-face, sends $\frac{1}{2}$ to each incident $4^{+}$-face. Thus, 
\begin{equation*}
\mu'(v) \geq 2\deg(v) - 6 - \left\lfloor \frac{2\deg(v)}{3} \right\rfloor \times 1 - \left\lceil \frac{\deg(v)}{3} \right\rceil \times \frac{1}{2} \geq 0. 
\end{equation*}

Let $f$ be a face in $\mathcal{N}$. If $f$ is a $3$-face in $\mathcal{N}$, then it receives $2$ from $D$ and receives $1$ from each incident internal vertex, so $\mu'(f) \geq -3 + 2 + 1 = 0$. If $f$ is a $4$-face in $\mathcal{N}$, then $\mu'(f) \geq -2 + 2 = 0$ by \autoref{Remark-3}. If $f$ is a $5^{+}$-face in $\mathcal{N}$, then $\mu'(f) \geq -1 + 2 = 1 > 0$ by \autoref{Remark-3}. 

If $f$ is an internal $3$-face, then it receives $1$ from each incident vertex, and then $\mu'(f) = -3 + 3 \times 1 = 0$. If $f$ is an internal $4$-face, then it is adjacent to four $4^{+}$-faces, and it receives $\frac{1}{2}$ from each incident vertex, which implies that $\mu'(f) = -2 + 4 \times \frac{1}{2} = 0$. If $f$ is an internal $6^{+}$-face, then $\mu'(f) \geq \mu(f) = \deg(f) - 6 \geq 0$. 

Let $f = v_{1}v_{2}v_{3}v_{4}v_{5}$ be an internal $5$-face. Since \autoref{fig:subfig:b-} is forbidden in $G$, $f$ is adjacent to at most two $3$-faces. By symmetry, we may assume that neither $v_{1}v_{5}$ nor $v_{4}v_{5}$ is adjacent to a $3$-face. If $f$ is adjacent to a $3$-face $uvw$ with $uv$ on $f$, then at most one of $u$ and $v$ is a $4_{2}$-vertex, for otherwise \autoref{fig:subfig:a-} or \autoref{fig:subfig:b-} will appear in $G$. Hence, at most two of $\{v_{1}, v_{2}, v_{3}, v_{4}\}$ are $4_{2}$-vertices. By the discharging rules, if $x \in \{v_{1}, v_{2}, v_{3}, v_{4}\}$ but it is not a $4_{2}$-vertex, then it sends at least $\frac{1}{4}$ to $f$. Moreover, whenever $v_{5}$ is a $4_{0}$-, $4_{1}$- or $5^{+}$-vertex, it sends $\frac{1}{2}$ to $f$. This implies that $\mu'(f) \geq -1 + \frac{1}{2} + 2 \times \frac{1}{4} = 0$ by \ref{R1-}\ref{R1b-}, \ref{R1-}\ref{R1c-} and \ref{R2-}. 

Finally, we show that $D$ has a positive final charge $\mu^{*}$. By \ref{R3-} and \ref{TOD}, the final charge of the outer face $D$ is 
\begin{equation*}
\mu^{*}(D) = \deg(D) + 6 + \sum_{v \in V(D)} \big(2\deg(v) - 6\big) - 2\left(\sum_{v \in V(D)} \big(\deg(v) - 2\big)\right) + p = 6 - \deg(D) + p, 
\end{equation*}
where $p$ is the total charge sent to $D$ by \ref{TOD}. 
So we may assume that $D$ is a $6$- or $7$-face, otherwise we are done. 

\begin{lemma}\label{4Plus}
Every $4^{+}$-face in $\mathcal{N}$ sends at least $1$ to $D$ by \ref{TOD}. 
\end{lemma}
\begin{proof}
After applying \ref{R1-}, \ref{R2-} and \ref{R3-}, every $5^{+}$-face in $\mathcal{N}$ has charge at least $-1 + 2 = 1$, thus it sends at least $1$ to $D$ by \ref{TOD}. Suppose that $f$ is a $4$-face in $\mathcal{N}$. Since $\mathcal{C}$ is an induced cycle and $S \neq V(G)$, there is at least one internal vertex incident with $f$. Moreover, there are at least two internal vertices incident with $f$, for otherwise it contradicts \autoref{L1-MR3}\ref{Le-MR3}. Since $f$ is a $4$-face, it is adjacent to four $4^{+}$-faces. Therefore, each internal vertex on $f$ sends $\frac{1}{2}$ to $f$, which implies that $f$ sends at least $4 - 6 + 2 + 2 \times \frac{1}{2} = 1$ to $D$ by \ref{TOD}. 
\end{proof}

If there is a $6^{+}$-face $g$ in $\mathcal{N}$, then $g$ sends at least $2$ to $D$ by \ref{TOD}, which implies that $\mu^{*}(D) \geq 6 - \deg(D) + 2 > 0$. So every face in $\mathcal{N}$ is a $5^{-}$-face. By \autoref{4Plus}, if there are at least two $4^{+}$-faces in $\mathcal{N}$, then $\mu^{*}(D) \geq 6 - \deg(D) + 2 \times 1 > 0$. So we may further assume that there is at most one $4^{+}$-face in $\mathcal{N}$. It follows that $D$ is adjacent to a $3$-face $f_{1}=uvw$, where $w$ is an internal vertex. Suppose that each of $wu$ and $wv$ is adjacent to exactly one $3$-face. Thus, each of $wu$ and $wv$ is adjacent to a $5$-face in $\mathcal{N}$, this contradicts that there is at most one $4^{+}$-face in $\mathcal{N}$. Suppose that $wu$ is adjacent to another $3$-face $wuv'$. If $v'$ is on $\mathcal{C}$, then each of $wv$ and $wv'$ is adjacent to a $5$-face in $\mathcal{N}$, but this is impossible. If $v'$ is an internal vertex, then each of $wv$ and $uv'$ is adjacent to a $5$-face in $\mathcal{N}$, but this is also impossible. This completes the proof. 
\end{proof}

\vskip 0mm \vspace{0.3cm} \noindent{\bf Acknowledgments.} This work was supported by the Fundamental Research Funds for Universities in Henan (YQPY20140051).

\end{document}